\documentclass[12pt]{article}
 
\usepackage{amstext,amssymb,amsmath,stmaryrd,amsthm}
\usepackage{geometry} 
\usepackage{color}
\usepackage{hyperref}
\usepackage{makeidx}
\makeindex

\numberwithin{equation}{section}

\newtheorem{theorem}{Theorem}[section]
\newtheorem{corollary}[theorem]{Corollary}
\newtheorem{lemma}[theorem]{Lemma}
\newtheorem{proposition}[theorem]{Proposition}
\newtheorem{definition}[theorem]{Definition}
\newtheorem{rem}[theorem]{Remark}

\newcommand{\R}{\mathbb{R}}
\newcommand{\He}{\mathbb{H}}
\newcommand{\Z}{\mathbb{Z}}
\newcommand{\E}{\mathbb{E}}
\newcommand{\Ind}{\mathbf{1}}

\begin{document}

\title{\Large{Logarithmic Sobolev inequalities for infinite dimensional H\"ormander type generators on the Heisenberg group}
\thanks{
{Supported by EPSRC 
EP/D05379X/1}} 
 }

\author{ J. Inglis , I. Papageorgiou  \\
{\small{Imperial College London}}\\
}

\maketitle

{\small{
\noindent\textbf{Abstract}: 
{\em The Heisenberg group is one of the simplest sub-Riemannian settings in which we can define non-elliptic H\"ormander type generators.  We can then consider coercive inequalities associated to such generators.  
We prove that a certain class of non-trivial Gibbs measures with quadratic interaction potential on an infinite product of Heisenberg groups satisfy logarithmic Sobolev inequalities. 
}}\\


%

\tableofcontents


\section{Introduction}
\label{intro}
Ever since P. Federbush (see \cite{Fed}) and   L. Gross (see \cite{G}) proved that a logarithmic Sobolev inequality is equivalent to hypercontractivity of the associated semigroup (see \cite{G}), these inequalities have been the subject of much research and interest.  They have proved extremely useful as a tool in the control of the rate of convergence to equilibrium of spin systems, and were extensively studied (see for example \cite{B-H1},\cite{G-Z}, \cite{Led}, \cite{O-R}, \cite{S-Z1},\cite{Yo2}, \cite{Ze}).  Up until recently, however, most of the attention has been focused on the case of elliptic generators, for which there are some very powerful methods for proving such inequalities.  Our aim here is to show that a certain class of infinite dimensional measures corresponding to non-elliptic H\"ormander type generators satisfy logarithmic Sobolev inequalities.

One method that exists for proving coercive inequalities such as the logarithmic Sobolev inequality and the spectral gap inequality, as well as gradient bounds (which are closely related) involves showing that the so-called $CD(\rho, \infty)$ condition holds (see \cite{A-B-C-F-G-M-R-S}, \cite{Bakry}).  Indeed, let $L$ be the generator of a Markov semigroup $P_t$, and define the operators

\begin{align*}
\Gamma(f,f) &= \frac{1}{2}\left(L(f^2) - 2fLf\right)\\
\Gamma_2(f,f) &= \frac{1}{2}\left[L\Gamma(f,f) - 2\Gamma(f, Lf)\right].
\end{align*}
We say that the $CD(\rho, \infty)$ holds when there exists $\rho\in\R$ such that
\[
\Gamma_2(f,f) \geq \rho \Gamma(f,f).
\]
When $L$ is elliptic such a condition holds in many situations.  In the case when $M$ is a complete connected Riemannian manifold, and $\nabla$ and $\Delta$ are the standard Riemannian gradient and Laplace-Beltrami operators, taking $L=\Delta$, the condition reads
\[
|\nabla\nabla f|^2 + \mathrm{Ric}(\nabla f, \nabla f) \geq \rho|\nabla f|^2.
\]
This holds for some $\rho\in\R$ when $M$ is compact, or for $\rho =0$ when $M=\R^n$ with the usual metric, since $\mathrm{Ricci}=0$.

However, in this paper we will consider non-elliptic H\"ormander generators. For such generators these methods do not work, since the $CD(\rho, \infty)$ condition does not hold.  Indeed, the Ricci tensor of our generators can be thought of as being $-\infty$ almost everywhere.  

We consider an $N$-dimensional lattice and impose interactions between points in the lattice described by an unbounded quadratic potential.  In the standard case where the underlying space is Euclidean, the $CD(\rho, \infty)$ condition allows us to prove that the finite dimensional measures on the lattice, which depend on the boundary conditions, satisfy logarithmic Sobolev inequalities uniformly on the boundary conditions.  It is then possible to pass to the infinite dimensional measure.  We aim for a comparable result in a more complicated sub-Riemannian setting, using different methods.

In \cite{L-Z} a similar situation is studied, in that the authors consider a system of H\"ormander generators in infinite dimensions and prove logarithmic Sobolev inequalities as well as some ergodicity results.  The main difference between the present set up and their situation is that we consider a non-compact underlying space, namely the Heisenberg group, in which the techniques of \cite{L-Z} cannot be applied.

\section{Logarithmic Sobolev inequalities on the Heisenberg group}
\label{Logarithmic Sobolev inequalities on the Heisenberg group}
We consider the Heisenberg group, $\mathbb{H}$, which can be described as $\mathbb{R}^3$ with the following group operation:
$$x\cdot\tilde{x} = (x_1, x_2, x_3)\cdot(\tilde{x}_1, \tilde{x}_2, \tilde{x}_3) = (x_1 + \tilde{x}_1, x_2 + \tilde{x}_2, x_3 + \tilde{x}_3 + \frac{1}{2}(x_1\tilde{x}_2 - x_2\tilde{x}_1)).$$
$\mathbb{H}$ is a Lie group, and its Lie algebra $\mathfrak{h}$ can be identified with the space of left invariant vector fields on $\mathbb{H}$ in the standard way.  By direct computation we see that this space is spanned by 
\begin{eqnarray}
X_1 &=& \partial_{x_1} - \frac{1}{2}x_2\partial_{x_3} \nonumber \\
X_2 &=& \partial_{x_2} + \frac{1}{2}x_1\partial_{x_3} \nonumber \\
X_3 &=& \partial_{x_3} = [X_1, X_2]. \nonumber
\end{eqnarray}
From this it is clear that $X_1, X_2$ satisfy the H\"ormander condition (i.e. $X_1, X_2$ and their commutator $[X_1, X_2]$ span the tangent space at every point of $\He$).  It is also easy to check that the left invariant Haar measure (which is also the right invariant measure since the group is nilpotent) is the Lebesgue measure $dx$ on $\R^3$. 

$\He$ is naturally equipped with a 1-parameter family of automorphisms $\{\delta_\lambda\}_{\lambda>0}$ defined by
\[
\delta_\lambda(x) := \left(\lambda x_1, \lambda x_2, \lambda^2 x_3\right).
\]
$\{\delta_\lambda\}_{\lambda>0}$ is called a family of \textit{dilations}.  Thus $\He$ is an example of a homogeneous Carnot group (see \cite{B-L-U} for an extensive study of such groups).
 
On $C^\infty_0(\He)$, define the \textit{sub-gradient} to be the operator given by
$$\nabla := (X_1,X_2)$$ 
and the \textit{sub-Laplacian} to be the second order operator given by 
$$\Delta := X_1^2 + X_2^2.$$
$\nabla$ can be treated as a closed operator from $L^2(\mathbb{H}, dx)$ to $L^2(\mathbb{H}; \mathbb{R}^2, dx)$.  Similarly, since $\Delta$ is densely defined and symmetric in $L^2(\He, dx)$, we may treat $\Delta$ as a closed self-adjoint operator on $L^2(\He, dx)$ by taking the Friedrich extension.  

We introduce the logarithmic Sobolev inequality on $\He$ in the following way.

\begin{definition}
Let $q\in(1,2]$, and let $\mu$ be a probability measure on $\He$.  $\mu$ is said to satisfy a \textit{$q$-logarithmic Sobolev inequality} $(LS_q)$ on $\He$ if there exists a constant $c>0$ such that for all smooth functions $f: \He\to \R$
\begin{equation}
\label{LSq}
\mu\left(|f|^q\log\frac{|f|^q}{\mu|f|^q}\right) \leq c\mu\left(|\nabla f|^q\right)
\end{equation}
where $\nabla$ is the sub-gradient on $\He$.
\end{definition}

\begin{rem}
The $(LS_q)$ was introduced in \cite{B-L} and further studied in \cite{B-Z}, as a variation of the more standard $(LS_2)$ inequality.  Here it is noted that for $q<2$, the $(LS_q)$ inequality serves as a certain sharpening of $(LS_2)$, at least when the underlying space is finite dimensional.
\end{rem}

\begin{rem}
\label{tensorisation}
We recall four important standard properties of $(LS_q)$ inequalities that will be used below (see \cite{B-Z} and \cite{G-Z}):
\begin{itemize}
\item[\rm{(i)}] {\rm$(LS_q)$ is stable under tensorisation:} Suppose $\mu_1$ and $\mu_2$ satisfy $(LS_q)$ inequalities with constants $c_1$ and $c_2$ respectively.  Then $\mu_1\otimes\mu_2$ satisfies an $(LS_q)$ inequality with constant $\max\{c_1, c_2\}$.
\item[\rm{(ii)}] {\rm$(LS_q)$ is stable under bounded perturbations:} Suppose $d\mu = \frac{e^{-U}}{Z}dx$ satisfies an $(LS_q)$, and that $W$ is bounded.  Then $\tilde\mu(dx) = \frac{e^{-U - W}}{\tilde Z}dx$ satisfies an $(LS_q)$ inequality.
\item[\rm{(iii)}] {\rm$(LS_q) \Rightarrow (SG_q)$:} Suppose $\mu$ satisfies an $(LS_q)$ inequality with constant $c$.  Then $\mu$ satisfies a $q$-spectral gap inequality (we say $\mu$ satisfies an $(SG_q)$ inequality) with constant $\frac{4c}{\log2}$ i.e. 
\[
\mu\left|f - \mu f\right|^q \leq \frac{4c}{\log2} \mu\left(|\nabla f|^q\right)
\]
for all smooth $f$.
\item[\rm{(iv)}] When the underlying space is finite dimensional, $(LS_q) \Rightarrow (LS_{q'})$ and $(SG_q) \Rightarrow (SG_{q'})$ for $q<q'$.
\end{itemize}
\end{rem} 

The operator $\Delta$ is non-elliptic, but by H\"ormander's theorem it is hypoelliptic, so that the associated heat semigroup $P_t$ has a smooth convolution kernel with respect to the Haar measure.  Note that we can calculate the $\Gamma$ and $\Gamma_2$ functions for this generator explicitly, and we can easily see that there does not exist a constant $\rho\in\R$ such that $\Gamma_2\geq\rho\Gamma$.  However, despite this, in \cite{B-B-B-C} and \cite{Li} it was recently shown that there exists a constant $C$ such that
\[
|\nabla P_t f|(x) \leq CP_t(|\nabla f|)(x), \qquad \forall x\in \He,
\]
which is a surprising result.  It follows directly from this gradient bound that the heat kernel measure on $\He$ satisfies an $(LS_2)$ inequality on $\He$.

\begin{rem}
Since we have the sub-gradient on the right-hand side, \eqref{LSq} is a logarithmic Sobolev inequality corresponding to a H\"ormander type generator.  Indeed,  if $\mu(dx) = \frac{e^{-U}}{Z}dx$ then it is clear that $\mathcal{L} = \Delta - \nabla U.\nabla$ is a Dirichlet operator satisfying
\[
\mu\left(f\mathcal{L}f\right) = - \mu\left(|\nabla f|^2\right),
\]
where $\Delta$ is the sub-Laplacian, and $\nabla$ the sub-gradient.
\end{rem}

Independently, and by very different methods, in \cite{H-Z} the authors were able to show that a related class of measures on $\He$ satisfy $(LS_q)$ inequalities (see Theorem \ref{hz} below).  To describe these we first need to introduce the natural distance function on $\He$, which is the so-called \textit{Carnot-Carath\'eodory distance}.  This distance is more natural than the usual Euclidean one, since it takes into account the extra structure that the Heisenberg group posseses.  

We define the Carnot-Carath\'eodory distance between two points in $\He$ by considering only \textit{admissible} curves between them.  A Lipschitz curve $\gamma:[0,1]\to\He$ is said to be \textit{admissible} if $\gamma'(s) = a_1(s)X_1(\gamma(s)) + a_2(s)X_2(\gamma(s))$ almost everywhere with measurable coefficients $a_1, a_2$ i.e. if $\gamma'(s) \in sp\{X_1(\gamma(s)),X_2(\gamma(s))\}$ a.e.  Then the length of $\gamma$ is given by
$$l(\gamma) = \int_0^1\left(a_1^2(s) + a_2^2(s)\right)^{1/2}ds$$
and we define the Carnot-Carath\'eodory distance between two points $x,y\in\He$ to be 
$$d(x,y) := \inf\{l(\gamma):\gamma\ \textrm{is an admissible path joining}\ x\ \textrm{to}\ y\}.$$
Write $d(x)=d(x,e)$, where $e$ is the identity.

\begin{rem}
This distance function is well defined as a result of Chow's theorem, which states that every two points in $\He$ can be joined by an admissible curve (see for example \cite{B-L-U},\cite{Grom}).
\end{rem}
 
It is well know that $d$ is a homogeneous norm on $\He$ i.e. $d(\delta_\lambda(x)) = \lambda d(x)$ for all $\lambda>0$, where $\delta_\lambda$ is the dilation as defined above (see for example \cite{B-L-U}).
 
Geodesics are smooth, and are helices in $\R^3$.  They have an explicit parameterisation.  For details see \cite{B-B-B-C}, \cite{Beals}, \cite{Bel}, \cite{Monti}.
We also have that $x=(x_1, x_2, x_3)\mapsto d(x)$ is smooth for $(x_1, x_2) \neq 0$, but not at points $(0,0, x_3)$, so that the unit ball has singularities on the $x_3$-axis (one can think of it as being 'apple' shaped).

In our analysis, we will frequently use the following two results.  The first is the well-known fact that the Carnot-Carath\'eodory distance satisfies the eikonal equation (see for example \cite{Monti}):

\begin{proposition}
\label{eik}
Let $\nabla$ be the sub-gradient on $\He$.  Then $|\nabla d(x)|=1$ for all $x=(x_1, x_2, x_3)\in\He$ such that $(x_1, x_2)\neq0$.
\end{proposition}

We must be careful in dealing with the notion of $\Delta d$, since it will have singularities on the $x_3$-axis.  However, the following (proved in \cite{H-Z}) provides some control of these singularities.

\begin{proposition}
\label{sub-laplacian lem}
Let $\Delta$ be the sub-Laplacian on $\He$. There exists a constant $K$ such that $\Delta d \leq \frac{K}{d}$ in the sense of distributions.
\end{proposition}

\begin{proof}
For the sake of completeness, we recall part of the proof given in \cite{H-Z}.  It suffices to show that $\Delta d \leq K$ on $\{d(x) =1\}$.  Indeed, using dilations and homogeneity, we have that 
\[
\Delta d(x)=\lambda\Delta d(\delta_\lambda(x))
\]
for all $x\neq0, \lambda>0$, so that for any $x\in\He\backslash\{0\}$
\begin{equation}
\label{laplacian estimate}
\Delta d(x) \leq \frac{1}{d(x)}\sup_{\{d(y) = 1\}}\Delta d(y).
\end{equation}
Since everything is smooth away from the $x_3$ axis, in order to prove that  $\Delta d \leq K$ on $\{d(x) =1\}$, it suffices to look at what happens in a small neighbourhood of $(0,0, z)$, where $z$ is such that $d\left((0,0,z)\right)=1$.
To do this, let
\[
A_\eta := \left\{(r,s) \in \R^2: s>0, r>-\eta s\right\},
\]
 and for $x = (x_1, x_2, x_3)\in\He$ write $\|x\| := \left(x_1^2 + x_2^2\right)^{1/2}$.  Then it is shown that there exists $\eta>0$ and a smooth function $\psi(r,s)$ defined on $A_\eta$ such that for $x=(x_1, x_2, x_3)\in \He$,
\[
d(x) = \psi(\|x\|, |x_3|),
\]
and moreover that $\partial_r\psi<0$ when $r=0$.  One can then compute that
\begin{equation}
\label{sub-laplacian}
\Delta d(x) = \frac{1}{\|x\|}\partial_r\psi(\|x\|, |x_3|) + \partial_r^2\varphi(\|x\|, |x_3|) + \frac{\|x\|^2}{4}\partial_{s}\varphi(\|x\|, |x_3|).
\end{equation}
From \eqref{sub-laplacian} it follows that $\Delta d$ is bounded from above in a small neighbourhood of $(0,0,z)$, since although the first term is unbounded, it is negative.
\end{proof}

The following result is also found in \cite{H-Z}.

\begin{theorem}
\label{hz}
Let $\mu_p$ be the probability measure on $\He$ given by 
\[
\mu_p (dx) = \frac{e^{-\beta d^p(x)}}{\int_\He e^{-\beta d^p(x)}dx}dx
\]
where $p\geq2$, $\beta >0$, $dx$ is the Lebesgue measure on $\R^3$ and $d(x)$ is the Carnot-Carath\'edory distance.  Then $\mu_p$ satisfies an $(LS_q)$ inequality, where $\frac{1}{p}+\frac{1}{q} = 1$.
\end{theorem}

In the remainder of this paper we use the methods contained in \cite{H-Z}, together with an iterative procedure based on ideas contained in \cite{G-Z}, \cite{Led}, \cite{Ze2} and \cite{Ze} to prove an $(LS_q)$ inequality for a class of infinite dimensional measures on $(\He)^{\Z^N}$.

\section{Infinite dimensional setting and main result}
\label{Infinite dimensional setting}
\paragraph{\textit{The Lattice:}}
Let $\mathbb{Z}^N$ be the $N$-dimensional square lattice, for some fixed $N\in\mathbb{N}$.  We equip $\Z^N$ with the $l^1$ lattice metric $dist(\cdot,\cdot)$, defined by
\[
dist(i,j) := \sum_{l=1}^N|i_l - j_l|
\]
for $i=(i_1,\dots, i_N), j=(j_1, \dots, j_N) \in \Z^N$.  For $i,j\in\Z^N$ we will also write 
\[
i\sim j \qquad \Leftrightarrow \qquad dist(i,j) = 1
\]
i.e. $i\sim j$ when $i$ and $j$ are nearest neighbours in the lattice.

For $\Lambda \subset \mathbb{Z}^N$, we will write $|\Lambda|$ for the cardinality of $\Lambda$, and $\Lambda \subset \subset \mathbb{Z}^N$ when $|\Lambda|<\infty$.
  
\paragraph{\textit{The Configuration Space:}}
Let $\Omega = (\mathbb{H})^{\mathbb{Z}^N}$ be the \textit{configuration space}.  We introduce the following notation.  Given $\Lambda \subset \mathbb{Z}^N$ and $\omega = (\omega_i)_{i\in\mathbb{Z}^N}\in\Omega$, let $\omega_\Lambda := (\omega_i)_{i\in\Lambda} \in\mathbb{H}^\Lambda$ (so that $\omega\mapsto \omega_\Lambda$ is the natural projection of $\Omega$ onto $\mathbb{H}^\Lambda$).  
%

Let $f\colon \Omega \to \mathbb{R}$.  Then for $i\in\mathbb{Z}^N$ and $\omega \in \Omega$ define $f_i(\cdot|\omega)\colon\mathbb{H} \to \mathbb{R}$ by
$$f_i(x|\omega) := f(x\bullet_i\omega)$$
where the configuration $x\bullet_i\omega\in \Omega$ is defined by declaring its $i$th coordinate to be equal to $x\in\mathbb{H}$ and all the other coordinates coinciding with those of $\omega\in\Omega$.  Let $C^{(n)}(\Omega)$, $n\in\mathbb{N}$ denote the set of all functions $f$ for which we have $f_i(\cdot|\omega) \in C^{(n)}(\mathbb{H})$ for all $i\in\mathbb{Z}^d$ .  For $i\in\mathbb{Z}^N, k\in\{1,2\}$ and $f\in C^{(1)}(\Omega)$, define
$$X_{i,k}f(\omega) := X_kf_i(x|\omega)\vert_{x = \omega_i},$$
where $X_1, X_2$ are the left invariant vector fields on $\He$ defined in section \ref{Logarithmic Sobolev inequalities on the Heisenberg group}.

Define similarly $\nabla_if(\omega) := \nabla f_i(x|\omega)\vert_{x = \omega_i}$ and $\Delta_if(\omega) := \Delta f_i(x|\omega)\vert_{x = \omega_i}$ for suitable $f$, where $\nabla$ and $\Delta$ are the sub-gradient and the sub-Laplacian on $\He$ respectively.  For $\Lambda \subset \Z^N$, set $\nabla_\Lambda f = (\nabla_if)_{i\in\Lambda}$ and
$$|\nabla_\Lambda f|^q := \sum_{i\in\Lambda}|\nabla_if|^q.$$
We will write $\nabla_{\mathbb{Z^d}} = \nabla$, since it will not cause any confusion.

Finally, a function $f$ on $\Omega$ is said to be \textit{localised} in a set $\Lambda\subset\Z^N$ if $f$ is only a function of those coordinates in $\Lambda$.

\paragraph{\textit{Local Specification and Gibbs Measure:}}
Let $\Phi = (\phi_{\{i,j\}})_{\{i,j\}\subset\Z^N, i\sim j}$ be a family of $C^2$ functions such that $\phi_{\{i,j\}}$ is localised in $\{i,j\}$.  Assume that there exists an $M\in(0,\infty)$ such that $\|\phi_{\{i,j\}}\|_\infty \leq M$ and $\|\nabla_i\nabla_j\phi_{\{i,j\}}\|_\infty\leq M$ for all $i,j\in\Z^N$ such that $i\sim j$.  We say $\Phi$ is a bounded potential of range 1.  For $\omega\in\Omega$, define
$$H_\Lambda^\omega(x_\Lambda) = \sum_{\substack{\{i,j\}\cap\Lambda\neq\emptyset \\ i\sim j}}\phi_{\{i,j\}}(x_i, x_j),$$
for $x_\Lambda = (x_i)_{i\in\Lambda} \in \He^\Lambda$, where the summation is taken over couples of nearest neighbours $i\sim j$ in the lattice with at least one point in $\Lambda$, and where $x_i = \omega_i$ for $i\not\in \Lambda$.

Now let $(\E^{\omega}_\Lambda)_{\Lambda\subset\subset\Z^N,\omega\in\Omega}$ be the local specification defined by
\begin{equation}
\label{local spec}
\E^{\omega}_\Lambda(dx_\Lambda)=\frac{e^{-U^{\omega}_\Lambda(x_\Lambda)}}{\int e^{-U^{\omega}_\Lambda(x_\Lambda)}dx_\Lambda} dx_{\Lambda}\equiv\frac{e^{-U^{\omega}_\Lambda(x_\Lambda)}}{Z^\omega_\Lambda} dx_{\Lambda}
\end{equation}
where $dx_\Lambda$ is the Lebesgue product measure on $\He^\Lambda$ and 
\begin{equation}
\label{hamiltonian}
U^{\omega}_\Lambda (x_\Lambda) = \alpha\sum_{i\in\Lambda}d^p(x_i) + \varepsilon\sum_{\substack{ \{i,j\}\cap\Lambda \neq \emptyset \\ i\sim j}}(d(x_i) + \rho d(x_j))^{2} + \theta H_\Lambda^\omega(x_\Lambda),
\end{equation}
for $\alpha>0$, $\varepsilon, \rho, \theta \in\R$, and $p\geq2$, where as above  $x_i = \omega_i$ for $i\not\in \Lambda$.

\begin{rem}
\label{indices}
In the case when $p=2$, we must have that $\varepsilon> -\frac{\alpha}{2N}$ to ensure that $\int e^{-U_\Lambda}dx_\Lambda <\infty$.
\end{rem}

We define an infinite volume \textit{Gibbs measure} $\nu$ on $\Omega$ to be a solution of the (DLR) equation:
\[
\nu \E^{\cdot}_\Lambda f = \nu f
\]
for all bounded measurable functions $f$ on $\Omega$.  $\nu$ is a measure on $\Omega$ which has $\E^{\omega}_\Lambda$ as its finite volume conditional measures.

~

The main result of this paper is the following:

\begin{theorem}
\label{main}
Let $\nu$ be a Gibbs measure corresponding to the local specification defined by \eqref{local spec} and \eqref{hamiltonian}.  Let $q$ be dual to $p$ i.e. $\frac{1}{p}+\frac{1}{q} =1$ and suppose $\varepsilon\rho>0$, with $\varepsilon>-\frac{\alpha}{2N}$ if $p=2$. Then there exists $\varepsilon_0, \theta_0>0$ such that for $|\varepsilon|<\varepsilon_0$ and $|\theta|<\theta_0$, $\nu$ is unique and satisfies an $(LS_q)$ inequality i.e. there exists a constant $C$ such that
\[
\nu\left(|f|^q\log\frac{|f|^q}{\nu |f|^q}\right) \leq C\nu\left(\sum_i|\nabla_if|^q\right)
\]
for all $f$ for which the right-hand side is well defined.
\end{theorem}

We briefly mention some consequences of this result.  The first follows directly from Remark \ref{tensorisation} part (iii).

\begin{corollary}
Let $\nu$ be as in Theorem \ref{main}.  Then $\nu$ satisfies the $q$-spectral gap inequality.  Indeed
\[
\nu\left|f-\nu f\right|^q \leq \frac{4C}{\log2}\nu\left(\sum_i|\nabla_if|^q\right)
\]
where $C$ is as in Theorem \ref{main}.
\end{corollary}

The proofs of the next two can be found in \cite{B-Z}.

\begin{corollary}
Let $\nu$ be as in Theorem \ref{main} and suppose $f:\Omega \to \R$ is such that $\||\nabla f|^q\|_\infty  < 1$. Then
\[
\nu\left( e^{\lambda f}\right) \leq \exp\left\{\lambda \nu(f) +   \frac{C}{q^q(q-1)}\lambda^q\right\}
\]
for all $\lambda>0$ where $C$ is as in Theorem \ref{main}.  Moreover, by applying Chebyshev's inequality, and optimising over $\lambda$, we arrive at the following `decay of tails' estimate
\[
\nu\left\{\left|f - \int fd\nu\right| \geq h\right\} \leq 2\exp\left\{-\frac{(q-1)^p}{C^{p-1}}h^p\right\}
\]
for all $h>0$, where $\frac{1}{p} + \frac{1}{q} =1$.
\end{corollary}

\begin{corollary}
Suppose that our configuration space is actually finite dimensional, so that we replace $\Z^N$ by some finite graph $G$, and $\Omega = (\He)^G$.  Then Theorem \ref{main} still holds, and implies that if $\mathcal{L}$ is a Dirichlet operator satisfying
\[
\nu\left(f\mathcal{L} f\right) = -\nu\left(|\nabla f|^2\right),
\]
then the associated semigroup $P_t = e^{t\mathcal {L}}$ is ultracontractive.
\end{corollary}

\begin{rem}
In the above we are only considering interactions of range $1$, but we can easily extend our results to deal with the case where the interaction is of finite range $R$.
\end{rem}

\section{Results for the single site measure}
\label{results for single site measure}
The aim of this section is to show that the single site measures
\[
\E^{\omega}_{\{i\}}(dx_i) =: \E^{\omega}_i(dx_i) = \frac{e^{-U^{\omega}_i(x_i)}}{Z^{\omega}_i}dx_i, \qquad i\in\Z^N,
\]
satisfy an $(LS_q)$ inequality uniformly on the boundary conditions $\omega\in\Omega$.  We will often drop the $\omega$ in the notation for convenience.  The work is strongly motivated by the methods of Hebisch and Zegarlinski described in \cite{H-Z}.
\begin{theorem}
\label{uniform LSq}
Let $\frac{1}{q} +\frac{1}{p}=1$, and $\varepsilon\rho>0$ with $\varepsilon>-\frac{\alpha}{2N}$ if $p=2$.  Then there exists a constant $c$, independent of the boundary conditions $\omega \in \Omega$ such that
\[\E^{\omega}_i\left(|f|^q\log\frac{|f|^q}{\E_i^\omega|f|^q}\right) \leq c \E^{\omega}_i(|\nabla_if|^q)\]
for all smooth $f:\Omega \to \R$.
\end{theorem}

It will be convenient to work with alternative measures to the ones defined above.  Indeed, if we can prove uniform $(LS_q)$ inequalities for the single site measures when $\theta=0$ (so that we no longer have the the bounded interaction term in \eqref{hamiltonian}), then by Remark \ref{tensorisation} (ii), which states that $(LS_q)$ inequalities are stable under bounded perturbations, Theorem \ref{uniform LSq} will hold.  Moreover, it is clear that
\[
\frac{e^{-\alpha d^p(x_i) - \varepsilon\sum_{j:j\sim i}(d(x_i) + \rho d(\omega_j))^{2}}}{\int e^{-\alpha d^p(x_i) - \varepsilon\sum_{j:j\sim i}(d(x_i) + \rho d(\omega_j))^{2}}dx_i} = \frac{e^{-\tilde{U}^\omega_i}}{\int e^{-\tilde{U}^\omega_i}dx_i},
\]
where 
$$
\tilde{U}_i(x_i) = \alpha d^p(x_i) + 2N\varepsilon d^{2}(x_i) + 2\varepsilon\rho d(x_i)\sum_{j:j\sim i}d(\omega_j).
$$
It is therefore sufficient to work with the measures defined by $\tilde{\E}_i^\omega(dx_i) = (\tilde{Z}_i^\omega)^{-1}e^{-\tilde{U}_i^\omega}$, instead of $\E_i^\omega$.

The proof of the theorem will be in three steps.  We first prove the following inequality, designated a `$U$-bound' in \cite{H-Z}.

\begin{lemma}
\label{lem q U-Bound}
Let $\frac{1}{p}+\frac{1}{q} =1$ and suppose $\varepsilon\rho>0$, with $\varepsilon>-\frac{\alpha}{2N}$ if $p=2$.  Then the following inequality holds:
\[
\tilde{\E}_i^\omega\left(|f|^q\left(d^{p}  + d \sum_{j:j\sim i}d(\omega_j)\right)\right) \leq A\tilde{\E}_i^\omega|\nabla_if|^q + B\tilde\E_i^\omega |f|^q,
\]
for all smooth $f:\Omega\to\R$, and some constants $A, B\in(0,\infty)$ independent of $\omega$. 
\end{lemma}

\begin{proof}
Without loss of generality assume $f\geq0$.  By the Liebniz rule, we have
\begin{equation}
\label{leibniz}
(\nabla_i f)e^{-\tilde{U}_i} = \nabla_i(fe^{-\tilde{U}_i}) + f\nabla_i \tilde{U}_ie^{-\tilde{U}_i}. 
\end{equation}

Taking the inner product of both sides of this equation with $d(x_i)\nabla_id(x_i)$ and integrating yields
\begin{align*}
\int_{\He}fd\nabla_id.\nabla_i\tilde{U}_ie^{-\tilde{U}_i}dx_i &\leq \int_{\He}d|\nabla_id||\nabla_i f|e^{-\tilde{U}_i}dx_i - \int_{\He}d\nabla_id.\nabla_i\left(fe^{-\tilde{U}_i}\right)dx_i \\
&= \int_{\He}d|\nabla_i f|e^{-\tilde{U}_i}dx_i + \int_{\He}f\nabla_i\cdot(d\nabla_id)e^{-\tilde{U}_i}dx_i,
\end{align*}
where we have used integration by parts and Proposition \ref{eik}.  Now by Proposition \ref{sub-laplacian lem}, we have
\[
\nabla_i\cdot(d\nabla_id) = |\nabla_id|^2 + d\Delta_id \leq 1+K
\]
in terms of distributions.  Therefore we have
\begin{align*}
\int_{\He}fd\nabla_id.\nabla_i\tilde{U}_ie^{-\tilde{U}_i}dx_i &\leq  \int_{\He}d|\nabla_i f|e^{-\tilde{U}_i}dx_i + (1+K)\int_{\He}fe^{-\tilde{U}_i}dx_i.
\end{align*}

Replacing $f$ by $f^q$ in this inequality, and using Young's inequality, we arrive at
\begin{align}
\label{U-bound one}
\int_{\He}f^qd\nabla_id.\nabla_i\tilde{U}_ie^{-\tilde{U}_i}dx_i &\leq \frac{1}{\tau}\int_{\He}|\nabla_i f|^qe^{-\tilde{U}_i}dx_i + \frac{q}{p}\tau^{p-1}\int_{\He}f^qd^pe^{-\tilde{U}_i}dx_i \nonumber \\
& \qquad+ (1+K)\int_{\He}f^qe^{-\tilde{U}_i}dx_i,
\end{align}
for all $\tau>0$.

We now calculate that
\begin{align*}
\nabla_i d(x_i).\nabla_i\tilde{U}_i(x_i) &= p\alpha d^{p-1}(x_i) + 4N\varepsilon d(x_i) + 2\varepsilon\rho \sum_{j:j\sim i}d(\omega_j),
\end{align*}
almost everywhere, again using Proposition \ref{eik}.

For $\varepsilon\rho>0$, we therefore have that there exist constants $a_1, b_1 \in (0, \infty)$ such that
\begin{equation}
\label{lower estimate1}
d(x_i)\nabla_i d(x_i).\nabla_i\tilde{U}_i(x_i) \geq a_1\left(d^{p}(x_i)  + d(x_i) \sum_{j:j\sim i}d(\omega_j)\right) - b_1.
\end{equation}
This is clear if $\varepsilon>0$.  If $\varepsilon<0$ and $p>2$ then we use the fact that for any $\delta\in(0,1)$ there exists a constant $C(\delta)$ such that $d \leq \delta d^{p} + C(\delta)$.  If $p=2$, recall from Remark \ref{indices} that we must assume $\varepsilon>-\frac{\alpha}{2N}$, and the assertion follows.

Using the estimate \eqref{lower estimate1} in \eqref{U-bound one} and taking $\tau$ small enough, we see that there exist constants $A, B\in(0, \infty)$ independent of $\omega$  such that
\begin{align*}
&\int f^q\left(d^{p}(x_i)  +  d(x_i) \sum_{j:j\sim i}d(\omega_j)\right)e^{-\tilde{U}_i}dx_i \\
& \qquad \leq A\int |\nabla_if|^qe^{-\tilde{U}_i}dx_i + B\int f^qe^{-\tilde{U}_i}dx_i,
\end{align*}
which proves the lemma.
\end{proof}

The second step is to use this to prove that $\tilde\E_{i}^{\omega}$ satisfies a $q$-spectral gap inequality uniformly on the boundary conditions $\omega$.

\begin{lemma}
\label{qSG}
Let $\frac{1}{p} +\frac{1}{q}=1$ and suppose $\varepsilon\rho>0$, with $\varepsilon>-\frac{\alpha}{2N}$ if $p=2$.  Then $\tilde{\E}_i^\omega$ satisfies the q-spectral gap inequality uniformly on the boundary conditions i.e. there exists a constant $c_0\in(0,\infty)$ independent of $\omega$ such that
\[\tilde\E^{\omega}_i|f-\tilde\E^{\omega}_if|^q \leq c_0 \tilde\E^{\omega}_i|\nabla_i f|^q\]
for all smooth $f:\Omega\to\R$.
\end{lemma}

\begin{proof}
First note that
\begin{equation}
\label{note}
\tilde{\E}_i|f-\tilde{\E}_if|^q \leq 2^{q} \tilde{\E}_i|f-m|^q
\end{equation}
for any $m\in\R$.

Let $\mathcal{W}_i^\omega(x_i) := d^{p}(x_i)  +  d(x_i) \sum_{j:j\sim i}d(\omega_j)$.   Then, for any $L\in(0,\infty)$, we have
\begin{align*}
\tilde{\E}_i|f-m|^q & = \tilde{\E}_i|f - m|^q\Ind_{\{\mathcal{W}_i \leq L\}} + \tilde{\E}_i|f - m|^q\Ind_{\{\mathcal{W}_i \geq L\}}
\end{align*}
where $\Ind_{\{\mathcal{W}_i \leq L\}}$ is the indicator function of the set $A^\omega(L) := \{x_i \in \He: \mathcal{W}^\omega_i(x_i) \leq L\}$.  Let
\[
I_1:= \tilde{\E}_i|f - m|^q\Ind_{\{\mathcal{W}_i  \leq L\}}, \qquad I_2 :=\tilde{\E}_i|f - m|^q\Ind_{\{\mathcal{W}_i\geq L\}}.
\]
We estimate each of these terms separately.  We can treat $f$ as a function of $x_i$ only by fixing all the others.  Take 
\[
m = m(f) := \frac{1}{|A^\omega(L)|}\int_{A^\omega(L)} f(x_i)dx_i
\]
where $|A^\omega(L)| = \int_{A^\omega(L)}dx_i$ is the Lebesgue measure of $A^\omega(L)$.  Then we have that
\begin{align}
\label{I_1}
I_1 & = \int_{A^\omega(L)}\left| f(x_i) - \frac{1}{|A^\omega(L)|}\int_{A^\omega(L)}f(y_i)dy_i\right|^q\frac{e^{-\tilde U^\omega_i(x_i)}}{\tilde Z_i}dx_i\nonumber\\
&\leq \frac{e^{2N|\varepsilon| L^\frac{2}{p}}}{\tilde Z_i} \int_{A^\omega(L)}\left| f(x_i) - \frac{1}{|A^\omega(L)|}\int_{A^\omega(L)}f(y_i)dy_i\right|^qdx_i
\end{align}
since on $A^\omega(L)$ we have that $\tilde U^\omega_i \geq - 2N|\varepsilon|L^{\frac{2}{p}}$.  Now, using the invariance of the Lebesgue measure with respect to the group translation
\begin{align}
\label{interp1}
&\int_{A^\omega(L)}\left| f(x_i) - \frac{1}{|A^\omega(L)|}\int_{A^\omega(L)}f(y_i)dy_i\right|^qdx_i\nonumber\\
&\qquad \leq \frac{1}{|A^\omega(L)|^q} \int_{A^\omega(L)}\left(\int_\He|f(x_i) - f(x_iy_i)|\Ind_{A^\omega(L)}(x_iy_i)dy_i\right)^qdx_i\nonumber\\
&\qquad \leq \frac{1}{|A^\omega(L)|}\int_\He\int_\He|f(x_i) - f(x_iy_i)|^q \Ind_{A^\omega(L)}(x_iy_i)\Ind_{A^\omega(L)}(x_i)dy_idx_i
\end{align}
using H\"older's inequality.  Let $\gamma : [0, t] \to \He$ be a geodesic in $\He$ from $0$ to $y_i$ such that $|\dot\gamma(s)|\leq1$.  Then
\begin{align*}
|f(x_i) - f(x_iy_i)|^q &= \left| \int_0^t \frac{d}{ds}f(x_i\gamma(s))ds\right|^q \\
&=  \left| \int_0^t \nabla_i f(x_i\gamma(s)).\dot\gamma(s)ds\right|^q\\
&\leq t^\frac{q}{p}\int_0^t|\nabla_i f(x_i \gamma(s))|^qds
\end{align*}
again by H\"older's inequality, where $\frac{1}{p} +\frac{1}{q}=1$.  Here $t = d(y_i)$.  Using this estimate in \eqref{interp1} we see that
\begin{align}
\label{interp2}
&\int_{A^\omega(L)}\left| f(x_i) - \frac{1}{|A^\omega(L)|}\int_{A^\omega(L)}f(y_i)dy_i\right|^qdx_i\nonumber\\
&\qquad \leq \frac{1}{|A^\omega(L)|}\int_\He\int_\He d^\frac{q}{p}(y_i)\int_0^t|\nabla_i f(x_i \gamma(s))|^qds \Ind_{A^\omega(L)}(x_iy_i)\Ind_{A^\omega(L)}(x_i)dy_idx_i.
\end{align}
Note that when $x_iy_i\in A^\omega(L)$ and $x_i\in A^\omega(L)$ then we have $d(x_iy_i)\leq L^{\frac{1}{p}}$ and $d(x_i)\leq L^\frac{1}{p}$, so that
\[
d(y_i) = d(x_i^{-1}x_iy_i) \leq d(x_i) + d(x_iy_i) \leq 2L^\frac{1}{p}.
\]
Therefore, continuing \eqref{interp2}, 
\begin{align}
\label{interp3}
&\int_{A^\omega(L)}\left| f(x_i) - \frac{1}{|A^\omega(L)|}\int_{A^\omega(L)}f(y_i)dy_i\right|^qdx_i\nonumber\\
&\qquad \leq\frac{(2L^\frac{1}{p})^\frac{q}{p}}{|A^\omega(L)|}\int_\He\int_\He\int_0^t|\nabla_i f(x_i \gamma(s))|^q \Ind_{A^\omega(L)}(x_iy_i)\Ind_{A^\omega(L)}(x_i)dsdy_idx_i.
\end{align}

Next we note that for $x_iy_i\in A^\omega(L)$ and $x_i\in A^\omega(L)$ we have
\[
d^p(y_i) + d(y_i) \sum_{j:j\sim i}d(\omega_j) \leq 2^pL
\]
and 
\begin{align*}
d^p(x_i\gamma(s)) + d(x_i\gamma(s)) \sum_{j:j\sim i}d(\omega_j) &\leq 2^{p-1}\left(d^p(x_i) + d(x_i) \sum_{j:j\sim i}d(\omega_j)\right) \\
&\quad + 2^{p-1}\left(d^p(\gamma(s)) + d(\gamma(s)) \sum_{j:j\sim i}d(\omega_j)\right)\\
&\leq 2^{p-1}L + 2^{p-1}\left(d^p(y_i) + d(y_i) \sum_{j:j\sim i}d(\omega_j)\right)\\
&\leq 2^{p-1}L(1+2^p) =: R.
\end{align*}
Thus we can continue \eqref{interp3} by writing
\begin{align}
\label{interp4}
&\int_{A^\omega(L)}\left| f(x_i) - \frac{1}{|A^\omega(L)|}\int_{A^\omega(L)}f(y_i)dy_i\right|^qdx_i\nonumber\\
&\quad \leq\frac{(2L^\frac{1}{p})^\frac{q}{p}}{|A^\omega(L)|}\int_\He\int_\He\int_0^t|\nabla_i f(x_i \gamma(s))|^qds \Ind_{A^\omega(R)}(x_i\gamma(s))\Ind_{A^\omega(2^pL)}(y_i)dy_idx_i\nonumber\\
& \quad \leq\frac{(2L^\frac{1}{p})^\frac{q}{p}}{|A^\omega(L)|}\int_\He d(y_i) \left(\int_\He|\nabla_i f(x_i)|^q\Ind_{A^\omega(R)}(x_i)dx_i\right)\Ind_{A^\omega(2^pL)}(y_i)dy_i\nonumber\\
& \quad \leq\frac{(2L^\frac{1}{p})^{\frac{q}{p}+1}}{|A^\omega(L)|}\int_\He \left(\int_\He|\nabla_i f(x_i)|^q\Ind_{A^\omega(R)}(x_i)dx_i\right)\Ind_{A^\omega(2^pL)}(y_i)dy_i\nonumber\\
& \quad = 2^qL^\frac{q}{p}\frac{|A^\omega(2^pL)|}{|A^\omega(L)|}\int_{A^\omega(R)}|\nabla_i f(x_i)|^qdx_i\nonumber\\
& \quad \leq  2^qL^\frac{q}{p}e^R\frac{|A^\omega(2^pL)|}{|A^\omega(L)|}\int |\nabla_i f(x_i)|^qe^{-\tilde U_i^\omega}dx_i
\end{align}
where in the last line we have used the fact that on the set $A^\omega(R)$ we have $e^{-\tilde U_i^\omega} \geq e^{-R}$.  We finally note that $|A^\omega(2^pL)|/|A^\omega(L)|$ can be bounded above by a constant $C_1$ independent of $\omega$.  This is because $|A^\omega(2^pL)|/|A^\omega(L)|\geq 1$ and
\[
\frac{|A^\omega(2^pL)|}{|A^\omega(L)|} \to 1 \quad {\rm as} \quad \sum_{j:j\sim i}d(\omega_j) \to \infty.
\]
Then, using \eqref{interp4} in \eqref{I_1} yeilds
\[
I_1 \leq C_2 \tilde\E_i|\nabla_i f|^q
\]
where $C_2 = 2^qL^\frac{q}{p}e^{2N|\varepsilon|L^\frac{2}{p} + R}C_1$ is independent of $\omega$.

For the second term, we have that
\begin{align*}
I_2 &\leq \frac{1}{L}\tilde\E_i\left(|f -m|^q\mathcal{W}_i\right) \\
& \leq \frac{A}{L}\tilde\E_i\left|\nabla f \right|^q + \frac{B}{L}\tilde\E_i\left|f-m\right|^q
\end{align*}
where we have used Lemma \ref{lem q U-Bound}.  Putting the estimates for $I_1$ and $I_2$ together, we see that
\begin{align*}
\tilde\E_i\left|f-m\right|^q & \leq  \left(C_2 + \frac{A}{L}\right) \tilde\E_i|\nabla_i f|^q + \frac{B}{L}\tilde\E_i\left|f-m\right|^q \\
\Rightarrow \tilde\E_i\left|f-m\right|^q & \leq \frac{C_2 + \frac{A}{L}}{1-  \frac{B}{L}}\tilde\E_i|\nabla_i f|^q
\end{align*}
for $L>B$, where all constants are independent of $\omega$.  We can finally use this in \eqref{note} to get the result. 

\end{proof}

We can now prove Theorem \ref{uniform LSq}:
\begin{proof}[of Theorem  \ref{uniform LSq}]

Our starting point is the classical Sobolev inequality on the Heisenberg group for the Lebesgue measure (\cite{Var}):  there exists a $t>0$ such that
\begin{equation}
\label{CS}
\left(\int|f|^{1+t}dx_i\right)^{\frac{1}{1+t}} \leq a \int|\nabla_if|dx_i + b\int|f|dx_i,
\end{equation}
for some constants $a,b\in(0,\infty)$. Without loss of generality, we may assume that $f\geq0$.  Suppose also, to begin with, that $\tilde\E_i(f)=1$.  Now, if we set
\[
g\equiv \frac{fe^{-\tilde{U}_i}}{\tilde{Z}_i}
\]
then
\begin{equation}
\label{LS1}
\tilde\E_i(f\log f) = \int_\He g\log g dx_i + \tilde\E_i(f\tilde{U}_i) + \log \tilde{Z}_i.
\end{equation}
Now by Jensen's inequality
\begin{align*}
\int g\log gdx_i & = \frac{1}{t}\int g\log g^tdx_i  \\
& = \frac{1}{t}\int g\log \left(d^{1+t}g^t\right)dx_i - \frac{1+t}{t}\int g\log ddx_i \\
&\leq \frac{1+t}{t}\log\left(\int (dg)^{1+t}dx_i\right)^{\frac{1}{1+t}} +\frac{1+t}{t}\int g ddx_i \\
& \leq\frac{1+t}{t}\left(\int (dg)^{1+t}dx_i\right)^{\frac{1}{1+t}} +\frac{1+t}{t}\tilde\E_i(f d) \\
&\leq \frac{a(1 + t)}{t}\int |\nabla_i (dg)|dx_i + \frac{1 +t}{t}(b+1)\tilde\E_i(f d)\\
&\leq  \frac{a(1 + t)}{t}\int d|\nabla_i g|dx_i + \frac{1 +t}{t}(b+1)\tilde\E_i(f d) + \frac{a(1 + t)}{t},
\end{align*}
where we have used the classical Sobolev inequality \eqref{CS}, the fact that we have assumed $\tilde\E_i(f)=1$, and the elementary inequality $\log x\leq x$.  Hence by \eqref{LS1}
\begin{align}
\label{LS2}
\tilde\E_i(f\log f) &\leq \frac{a(1 + t)}{t}\int d\left|\nabla_i \left(\frac{fe^{-\tilde{U}_i}}{\tilde{Z}_i}\right)\right|dx_i + \tilde\E_i(f\tilde{U}_i) + \frac{1 +t}{t}(b+1)\tilde\E_i(f d)\nonumber\\
&\quad +  \frac{a(1 + t)}{t} + \log \tilde{Z}_i \nonumber\\
& \leq \frac{a(1 + t)}{t}\tilde\E_i(d |\nabla_i f|) + \frac{a(1 + t)}{t}\tilde\E_i(f d|\nabla_ i \tilde U_i|) + \tilde\E_i(f\tilde{U}_i)  \nonumber\\
& \qquad + \frac{1 +t}{t}(b+1)\tilde\E_i(fd^p) + 1 +  \frac{a(1 + t)}{t} + \log \tilde{Z}_i.
\end{align}

Now, since $\varepsilon\rho>0$ we have that $\tilde Z_i^\omega \leq C_3$ for some constant $C_3\in(0, \infty)$ independent of $\omega$.

Moreover, we can directly calculate that
\begin{align}
\label{est1}
d(x_i)|\nabla_i \tilde{U}^\omega_i|(x_i) &\leq \alpha pd^p(x_i) + 4N|\varepsilon|d^2(x_i) + 2\varepsilon\rho d(x_i)\sum_{j:j\sim i}d(\omega_j)\nonumber\\
&\leq  \left(\alpha p + 4N|\varepsilon|\right)d^p(x_i) + 2\varepsilon\rho d(x_i)\sum_{j:j\sim i}d(\omega_j) + 4N|\varepsilon|\nonumber\\
& \leq a_3\mathcal{W}_i^\omega(x_i) + b_3
\end{align}
almost everywhere, where $\mathcal{W}_i^\omega(x_i) = d^{p}(x_i) + d(x_i)\sum_{j:j\sim i}d(\omega_j)$ as in Lemma \ref{qSG}, and $a_3 = \max\{\alpha p + 4N|\varepsilon|, 2\varepsilon\rho\}$ and $b_3=4N|\varepsilon|$ are constants independent of $\omega$.
Similarly there exist constants $a_4, b_4\in[0, \infty)$ independent of $\omega$ such that
\begin{equation}
\label{est2}
\tilde{U}^\omega_i(x_i) \leq a_4\mathcal{W}_i^\omega(x_i) + b_4.
\end{equation}

We can substitute estimates \eqref{est1} and \eqref{est2} into \eqref{LS2}.  This yields
\begin{align}
\label{LS2.5}
\tilde\E_i(f\log f) &\leq c_1\tilde\E_i(d |\nabla_i f|) + c_2\tilde\E_i(f \mathcal{W}_i) + c_3
\end{align}
where 
\begin{align*}
c_1= &\frac{a(1 + t)}{t}, \qquad c_2 =  \frac{aa_3(1 + t)}{t} + a_4 + \frac{1+t}{t}(b+1)\\
&c_3 =  1 +  \frac{a(1 + t)}{t} + \log C_3 + b_3 + b_4
\end{align*}
are all independent of $\omega$.

Now, by fiirst replacing $f$ by $\frac{f}{\tilde\E_i f}$ and then $f$ by $f^q$ in \eqref{LS2.5}, after an application of Young's inequality we see that
\begin{align}
\label{LS3}
\tilde\E_i\left(f^q\log \frac{f^q}{\tilde\E_i f^q}\right) & \leq \frac{c_1}{q}\tilde\E_i(|\nabla_i f|^q) + \left(c_2 + \frac{c_1}{p}\right)\tilde\E_i\left(f^q\mathcal{W}_i\right) + c_3\tilde\E_i(f^q).
\end{align}

We recognise that the second term in \eqref{LS3} can be bounded using Lemma \ref{lem q U-Bound}.  Indeed, using this estimate

\begin{equation}
\label{GLS}
\tilde\E_i\left(f^q\log \frac{f^q}{\tilde\E_i f^q}\right)  \leq \tilde c_1\tilde\E_i(|\nabla_i f|^q) + \tilde c_2\tilde\E_i(f^q),
\end{equation}
where $\tilde c_1 = \frac{c_1}{q} + A\left(c_2 + \frac{c_1}{p}\right)$ and $\tilde c_2 = c_3 + B\left(c_2 + \frac{c_1}{p}\right)$. Thus $\tilde\E^\omega_i$ satisfies the generalised $(LS_q)$ inequality uniformly on the boundary conditions.

Finally, we have the $q$-Rothaus inequality (see \cite{B-Z}, \cite{Ro}), which states that
\[
\tilde\E_i\left(f^q\log\frac{f^q}{\tilde\E_if^q}\right) \leq \tilde\E_i\left(\left|f-\tilde\E_if\right|^q\log\frac{\left|f-\tilde\E_if\right|^q}{\tilde\E_i\left|f-\tilde\E_if\right|^q}\right) + 2^{q+1}\tilde\E_i\left|f-\tilde\E_if\right|^q.
\]
Using this together with \eqref{GLS} and the $q$-spectral gap inequality proved in Lemma \ref{qSG} we thus arrive at a constant $c$ independent of $\omega$ such that
\[
\tilde\E_i\left(f^q\log\frac{f^q}{\tilde\E_if^q}\right) \leq c\tilde\E_i(|\nabla_i f|^q),
\]
which proves Theorem \ref{uniform LSq}.
\end{proof}

\section{The logarithmic Sobolev inequality for Gibbs measures}

In this section we show how to pass from the uniform $(LS_q)$ inequality for the single site measures $\E^{\omega}_i$, to the $(LS_q)$ inequality for the corresponding Gibbs measure $\nu$ on the entire configuration space $\Omega = (\He)^{\Z^N}$.  In the more standard Euclidean model, this problem has been extensively studied in the case $q=2$ , for example in \cite{B-H1}, \cite{G-Z}, \cite{Led}, \cite{Marton} and more recently in \cite{O-R}, as well as in many of the afore mentioned papers.  The case $q<2$ was looked at in \cite{B-Z}.  The following argument is strongly related to these methods, though it is based on the work contained in \cite{Ze2} and \cite{Ze}.

We work in greater generality than is required for Theorem \ref{main}, though the results of section \ref{results for single site measure} show that in the specific case where the local specification is defined by \eqref{local spec} and \eqref{hamiltonian}, the hypotheses \textbf{(H0)} and \textbf{(H1)} below are satisfied.  Then Theorem \ref{main} follows as an immediate corollary of Theorem \ref{thm1}.

Consider a local specification $\{\E^{\omega}_\Lambda\}_{\Lambda\subset\subset\Z^N, \omega\in\Omega}$ defined by 
\begin{equation}
\label{gen local spec}
\E^{\omega}_\Lambda(dx_{\Lambda})=\frac{e^{-\sum_{i\in\Lambda}\phi(x_i)-\sum_{\{i,j\}\cap\Lambda\neq\emptyset, i\sim j}J_{ij}V(x_i,x_j)}dx_\Lambda} {Z^{\omega}_\Lambda}
\end{equation}
where $Z^{\omega}_\Lambda$ is the normalisation factor and the summation is taken over couples of nearest neighbours $i\sim j$ in the lattice with at least one point in $\Lambda$ and where $x_i = \omega_i$ for $i\not\in \Lambda$, as before.  We suppose that $|J_{ij}| \in [0, J_0]$ for some $J_0>0$.

~

We will work with the following hypotheses:
 
 \begin{itemize}
 \item[\textbf{(H0):}] The one dimensional single site measures $\mathbb{E}^{\omega}_i$ satisfy $(LS_q)$ with a constant $c$ which is independent of the boundary conditions $\omega$.  
 
~
 
\item[\textbf{(H1):}] The interaction $V$ is such that
$$
\left\Vert \nabla_i \nabla_j V(x_i,x_j) \right\Vert_{\infty}<\infty.
$$
\end{itemize}
 
\begin{rem}
In the situation where {\rm \textbf{(H1)}} is not satisfied, i.e. when the interaction potential grows faster than quadratically, a number of results have been obtained in \cite{Ioannis1} and \cite{Ioannis2} under some additional assumptions.
\end{rem}

\begin{theorem}
\label{thm1}
Suppose the local specification $\{\mathbb{E}^{\omega}_\Lambda\}_{\Lambda\subset\subset \Z^N,\omega \in
\Omega}$ defined by \eqref{gen local spec} satisifies {\rm\textbf{(H0)}} and {\rm \textbf{(H1)}}.  Then, for sufficiently small $J_0$, the corresponding infinite dimensional Gibbs measure $\nu$ is unique and satisfies the $(LS_q)$ inequality 
$$\nu \left(|f|^q \log\frac{|f|^q}{\nu |f|^q}\right)\leq C \nu (\left| \nabla f\right|^q)$$                               
for some positive constant $C$. 
\end{theorem}

For notational sake we only prove this result for the case $N=2$, but our methods are easily generalised.  Before proving Theorem \ref{thm1} we will present some useful lemmata.  
 
\subsection{Lemmata:}

Define the following sets
\begin{align*}
&\Gamma_0=(0,0)\cup\{j \in \mathbb{Z}^2 : dist(j,(0,0))=2m \text{\; for some \;}m\in\mathbb{N}\}, \\  
&\Gamma_1=\mathbb{Z}^2\smallsetminus\Gamma_0 .
\end{align*}
where $dist(\cdot, \cdot)$ is as in section \ref{Infinite dimensional setting}.  Note that $dist(i,j)>1$ for all $i,j \in\Gamma_k,k=0,1$ and $\Gamma_0\cap\Gamma_1=\emptyset$.  Moreover $\mathbb{Z}^2=\Gamma_0\cup\Gamma_1$.  As above, for the sake of notation, we will write $\mathbb{E}_{\Gamma_k}=\mathbb{E}_{\Gamma_k}^{\omega}$ for $k=0,1$. We will also define 
$$ \mathcal{P}:=\mathbb{E}_{\Gamma_1}\mathbb{E}_{\Gamma_{0}}.$$
 
\begin{lemma} 
\label{lem1}
If the local specification $\{\mathbb{E}^{\omega}_\Lambda\}_{\Lambda\subset\subset \Z^N,\omega \in
\Omega}$ satisfies {\rm \textbf{(H0)}} and {\rm \textbf{(H1)}}, then, for sufficiently small $J_0$, there exist constants $\tilde D>0$ and $\tilde\eta\in(0,1)$ such that
\begin{equation}
\label{first lem}
\nu\left\vert \nabla_{\Gamma_k}(\mathbb{E}_{\Gamma_{l}}f)
\right\vert^q \leq \tilde D\nu\left\vert \nabla_{\Gamma_k} f
\right\vert^q+\tilde\eta\nu\left\vert \nabla_{\Gamma_l} f
\right\vert^q\end{equation}
for $k,l\in\{0,1\}$ such that $k\neq l$.
\end{lemma}

 \begin{proof}   
 
For convenience, suppose $k=1$ and $l=0$. The case $k=0,l=1$ follows similarly. We can write 
\begin{align}
\label{lem1-1}
\mathcal{I} &:=\nu\left\vert \nabla_{\Gamma_1 }
(\mathbb{E}_{\Gamma_0}f) \right\vert ^q = \nu\sum_{i\epsilon \Gamma_1}\left\vert \nabla_{i }
(\mathbb{E}_{\Gamma_0}f) \right\vert ^q \nonumber \\
& \leq \nu\sum_{i\epsilon \Gamma_1}
\left\vert \nabla_{i }
(\mathbb{E}_{\{\sim i\}}f) \right\vert ^q \nonumber \\
&\leq2^{q-1}\nu\sum_{i\epsilon \Gamma_1}
\left\vert \mathbb{E}_{\{\sim i\}}\nabla_{i } f
\right\vert
^q+
2^{q-1}J_{0}^{q}\nu\sum_{i\epsilon \Gamma_1}
\left\vert\mathbb{E}_{\{\sim i\}}(f\mathcal{U}_i) \right\vert ^q
\end{align}
where above we have denoted $\{ \sim i\}=\{j:j \sim i\}$, $W_i=\sum_{j\in\{\sim i\}}\nabla_iV(x_{i}, x_{j})$ and 
$$\mathcal{U}_i=W_i-\mathbb{E}_{\{ \sim i\}}W_i.$$
Then
\begin{align}
\label{lem1-2}
\mathcal{I} &\leq2^{q-1}\nu\sum_{i\epsilon \Gamma_1}
\mathbb{E}_{\{\sim i\}}\left\vert \nabla_{i } f \right\vert^q + 2^{q-1}J_{0}^q\nu\sum_{i\epsilon\Gamma_1}\left\vert\mathbb{E}_{\{\sim i\}}(f-\mathbb{E}_{\{\sim i\}}f)\mathcal{U}_i \right\vert ^q \nonumber \\ 
&\leq 2^{q-1}\nu\sum_{i\epsilon \Gamma_1}
\mathbb{E}_{\{\sim i\}}\left\vert \nabla_{i } f
\right\vert^q+2^{q-1}J_{0}^q\nu\left(\sum_{i\epsilon \Gamma_1}\mathbb{E}_{\{\sim i\}}\left|f-\mathbb{E}_{\{\sim i\}}f\right|^q\left(\mathbb{E}_{\{\sim i\}}\left|\mathcal{U}_i\right|^p\right)^{q/p}\right)
\end{align}
using H\"older's inequality and the fact that $\E_{\{\sim i\}}\mathcal{U}_i =0$.  Since interactions occur only between nearest neighbours in the lattice, we have that no interactions occur between points of the set $\{\sim i\}$.  Hence the measure $\mathbb{E}_{\{\sim i\}}^{\omega}$ is the product measure of the single site measures i.e.  $\mathbb{E}_{\{\sim i\}}^{\omega}=\otimes_{j \in\{\sim i\}}\mathbb{E}^{\omega}_j$.
Moreover, by \textbf{(H0)}, all measures $\mathbb{E}^{\omega}_j, j \in\{\sim i\}$ satisfy the $(LS_q)$ inequality with a constant $c$ uniformly on the
boundary conditions.  Therefore, since the $(LS_q)$ inequality is stable under tensorisation (see Remark \ref{tensorisation} (i)), we have that the product measure $\mathbb{E}_{\{\sim i\}}^{\omega}$ also satisfies the $(LS_q)$ inequality with the same constant $c$.  By Remark \ref{tensorisation} (iii), it follows that $\mathbb{E}_{\{\sim i\}}^{\omega}$ also satisfies the $q$-spectral gap inequality with constant $c_0 = \frac{4c}{\log2}$.

Hence we have
\begin{equation}
\label{lem1-3}
\mathbb{E}_{\{\sim i\}}\left|f-\mathbb{E}_{\{\sim i\}}f\right|^q\leq c_0\mathbb{E}_{\{\sim i\}}\left\vert
\nabla_{\{\sim i\}} f\right\vert^q.
\end{equation}
Moreover, by Remark \ref{tensorisation} (iv), since $q<p$ and $\E_{\{\sim i\}}$ is a measure on a finite dimensional space, we have there exists a constant $\tilde c_0$ such that
\begin{align}
\label{lem1-4}
\mathbb{E}_{\{\sim i\}}\left|\mathcal{U}_i\right|^p=& \mathbb{E}_{\{\sim i\}}\left|W_i-\mathbb{E}_{\{\sim
i\}}W_i\right|^p\leq \tilde{c}_0\sum_{_{ j \in\{\sim i\}}}\mathbb{E}_{\{\sim i\}}\left\vert \nabla_{j} W_i\right\vert
^p \nonumber\\
&\leq \tilde{c}_0\sum_{_{ j \in\{\sim i\}}}\mathbb{E}_{\{\sim i\}}\left\vert \nabla_{j} \nabla_iV(x_{i}, x_{j})\right\vert
^p\leq 4\tilde{c}_0M^p\end{align}
where $M=\|\nabla_i\nabla_j V(x_i, \omega_j)\|_\infty < \infty$ by \textbf{(H1)}. 

If we combine \eqref{lem1-2}, \eqref{lem1-3} and \eqref{lem1-4} we obtain
 \begin{align*}
\nu\left\vert \nabla_{\Gamma_1 }
(\mathbb{E}_{\Gamma_0}f) \right\vert ^q &\leq 2^{q-1}\nu\left(\sum_{i\epsilon \Gamma_1}
\mathbb{E}_{\{\sim i\}}\left\vert \nabla_{i } f
\right\vert^q\right)  \\
&\qquad + 2^{q-1}c\left(4\tilde{c}_0\right)^{q/p}M^qJ^q_0\nu\left(\sum_{i\epsilon \Gamma_1}\mathbb{E}_{\{\sim i\}}\left\vert
\nabla_{\{\sim i\}} f\right\vert^q\right)\\
&\leq  2^{q-1}\nu\left(\sum_{i\epsilon \Gamma_1}
\left\vert \nabla_{i } f
\right\vert^q\right)  \\
&\qquad + 2^{q+1}c\left(4\tilde{c}_0\right)^{q/p}M^qJ^q_0\nu\left(\sum_{i\epsilon \Gamma_0}\left\vert
\nabla_{i} f\right\vert^q\right).
\end{align*}
Therefore, choosing $J_0$ sufficiently small so that  $ 2^{q+1}c\left(4\tilde{c}_0\right)^{q/p}M^qJ^q_0<1$, we see that
$$\nu
\left\vert \nabla_{\Gamma_1 }
(\mathbb{E}^{\Gamma_0}f) \right\vert ^q\leq \tilde D\nu\left\vert \nabla_{ \Gamma_1 } f
\right\vert^q+\tilde\eta\nu\left\vert\nabla_{ \Gamma_0} f\right\vert^q$$
 with $\tilde{D} =2^{q-1}$ and $\tilde\eta= 2^{q+1}c\left(4\tilde{c}_0\right)^{q/p}M^qJ^q_0<1$, as required.

 \end{proof}
 
 \begin{lemma}
 \label{lem2}
Suppose the local specification  $\{\mathbb{E}_{\Lambda}^{\omega}\}_{\Lambda\subset\subset \Z^N, \omega \in
\Omega}$ satisfies {\rm \textbf{(H0)}} and {\rm \textbf{(H1)}}, and let $W_i=\sum_{j\in\{\sim i\}}\nabla_iV(x_{i}, x_{j})$ be as in the previous lemma.  Then there exists a constant $\kappa$, independent of the boundary conditions, such that
\begin{align*}
\mathbb{E}_{\{\sim i\}}&\left(|f|^q;  W_{i}\right) \leq \left(\mathbb{E}_{\{\sim i\}}f^q\right)^{\frac{1}{p}}\left( \kappa\E_{\{\sim i \}}\left|\nabla_{\{ \sim i\}} f\right|^q\right)^\frac{1}{q},
\end{align*}
where $\frac{1}{p}+\frac{1}{q}=1$ and $\E_{\{\sim i \}}(g;h):= \E_{\{\sim i \}}(gh) - \E_{\{\sim i \}}(g)\E_{\{\sim i \}}(h)$ for any functions $g,h$.
\end{lemma}
 \begin{proof} 
 
 Without loss of generality, we may suppose that $f\geq0$.  Let $\hat\E_{\{\sim i\}}$ be an isomorphic copy of $\mathbb{E}_{\{\sim i\}}$. Then we have 
\begin{align*}
\mathbb{E}_{\{\sim i\}}(f^q;W_{i}) &=\frac{1}{2}\E_{\{\sim i\}}\otimes\hat\E_{\{\sim i\}}\left(\left(f^q- \hat f^{q}\right)(W_{i}-\hat W_i  )\right) \\
& = \frac{1}{2}\E_{\{\sim i\}}\otimes\hat\E_{\{\sim i\}}\left[\left(\int_0^1\frac{d}{ds}F_s^qds\right)\left(W_i - \hat W_i\right)\right]
\end{align*}
where $F_s = sf + (1-s)\hat f$ for $s\in[0,1]$.  Then
\begin{align}
\label{covlem1}
\mathbb{E}_{\{\sim i\}}(f^q;W_{i}) &= \frac{q}{2} \E_{\{\sim i\}}\otimes\hat\E_{\{\sim i\}}\left[\left(\int_0^1F_s^{q-1}ds\right)\left(f-\hat f\right)\left(W_i - \hat W_i\right)\right] \nonumber\\
&\leq \frac{q}{2}\left\{ \E_{\{\sim i\}}\otimes\hat\E_{\{\sim i\}}\left(\int_0^1F_s^{q-1}ds\right)^p\right\}^\frac{1}{p} \nonumber\\
& \quad \times \left\{ \E_{\{\sim i\}}\otimes\hat\E_{\{\sim i\}}\left|f-\hat f\right|^q\left|W_i - \hat W_i\right|^q\right\}^{\frac{1}{q}}.
\end{align}

Now by Jensen's inequality and convexity of the function $y\mapsto y^q$ we have
\begin{align}
\label{covlem1a}
\left\{ \E_{\{\sim i\}}\otimes\hat\E_{\{\sim i\}}\left(\int_0^1F_s^{q-1}ds\right)^p\right\}^\frac{1}{p} &\leq \left\{ \int_0^1\E_{\{\sim i\}}\otimes\hat\E_{\{\sim i\}}F_s^{q}ds\right\}^\frac{1}{p}\nonumber\\
&\leq  \left\{ \int_0^1\E_{\{\sim i\}}\otimes\hat\E_{\{\sim i\}}\left(sf^{q} + (1-s)\hat f^q\right)ds\right\}^\frac{1}{p}\nonumber\\
&= \left(\E_{\{\sim i\}} f^q\right)^{\frac{1}{p}}.
\end{align}

Moreover,
\begin{align}
\label{covlem2}
&\E_{\{\sim i\}}\otimes\hat\E_{\{\sim i\}}\left|f-\hat f\right|^q\left|W_i - \hat W_i\right|^q \nonumber \\
& \qquad \leq  2^q\E_{\{\sim i\}}\otimes\hat\E_{\{\sim i\}}\left|f-\E_{\{\sim i\}}f\right|^q\left|W_i - \hat W_i\right|^q.
 \end{align}

We have the following relative entropy inequality (see eg \cite{A-B-C-F-G-M-R-S}, \cite{D-S}): if $\mu$ is a probability measure then
\[
\mu(uv) \leq \frac{1}{\tau}\mu(u)\log\mu(e^{\tau v}) + \frac{1}{\tau}\mu\left(u\log\frac{u}{\mu (u)}\right), \qquad \forall \tau>0.
\]
Applying this to the right hand side of \eqref{covlem2} with $\mu = \E_{\{\sim i\}}\otimes\hat\E_{\{\sim i\}}$ we see that $\forall \tau>0$
\begin{align}
\label{covlem3}
&\E_{\{\sim i\}}\otimes\hat\E_{\{\sim i\}}\left|f-\hat f\right|^q\left|W_i - \hat W_i\right|^q \nonumber\\
&\qquad \qquad \leq \frac{2^q}{\tau}\E_{\{\sim i\}}\left|f-\E_{\{\sim i\}}f\right|^q\log \E_{\{\sim i\}}\otimes\hat\E_{\{\sim i\}}\left(e^{\tau|W_i - \hat W_i|^q}\right) \nonumber \\
& \quad\quad\quad + \frac{2^q}{\tau} \E_{\{\sim i\}}\left(\left|f-\E_{\{\sim i\}}f\right|^q\log\frac{\left|f-\E_{\{\sim i\}}f\right|^q}{\E_{\{\sim i\}} \left|f-\E_{\{\sim i\}}f\right|^q}\right).
\end{align}

Now, by the Herbst argument, see for example \cite{Helffer} or \cite{Ledconc} , and using both \textbf{(H0)} and \textbf{(H1)}, we have that for some $\tau>0$ there exists a constant $\Theta>0$ independent of $\omega$ such that
\[
\E_{\{\sim i\}}\otimes\hat\E_{\{\sim i\}}\left(e^{\tau|W_i - \hat W_i|^q}\right) \leq \Theta.
\]
We can also use \textbf{(H0)} to bound the second term of \eqref{covlem3}.  This gives
\begin{align}
\label{covlem4}
\E_{\{\sim i\}}\otimes\hat\E_{\{\sim i\}}\left|f-\hat f\right|^q\left|W_i - \hat W_i\right|^q \nonumber &\leq \frac{2^q\log\Theta}{\tau}\E_{\{\sim i\}}\left|f-\E_{\{\sim i\}}f\right|^q \nonumber \\
& \quad\quad + \frac{2^qc}{\tau}\E_{\{\sim i\}}\left|\nabla_{\{\sim i\}}f\right|^q\nonumber\\
&\leq \frac{2^q}{\tau}\left(c_0\log\Theta + c\right) \E_{\{\sim i\}}\left|\nabla_{\{\sim i\}}f\right|^q
\end{align}
where $c_0=\frac{4c}{\log2}$ as above, by Remark \ref{tensorisation} (iii).

Putting estimates \eqref{covlem1a} and \eqref{covlem4} into \eqref{covlem1} we see that 
\begin{align*}
\mathbb{E}_{\{\sim i\}}(f^q;W_{i}) &\leq \left(\E_{\{\sim i\}} f^q\right)^{\frac{1}{p}}\left(\frac{q^q}{\tau}(c_0\log\Theta +c)\E_{\{\sim i\}}\left|\nabla_{\{\sim i\}}f\right|^q\right)^{\frac{1}{q}}
\end{align*}
which gives the desired result.
\end{proof}

\begin{lemma}
\label{lem3}
Suppose the local specification  $\{\mathbb{E}_{\Lambda}^{\omega}\}_{\Lambda\subset\subset \Z^N,\omega \in
\Omega}$ satisfies {\rm \textbf{(H0)}} and {\rm \textbf{(H1)}}.  Then, for sufficiently small $J_0$, there exist constants $D >0$ and $\eta\in(0,1)$ such that
\begin{equation}
\label{lem3ineq}
\nu\left\vert \nabla_{\Gamma_k}(\mathbb{E}_{\Gamma_{l}}|f|^q)^\frac{1}{q}
\right\vert^q \leq D\nu\left\vert \nabla_{\Gamma_k} f
\right\vert^q+\eta\nu\left\vert \nabla_{\Gamma_l} f
\right\vert^q\end{equation} 
for $k,l\in\{0,1\}, k\neq l$. 
\end{lemma}

\begin{proof} 
Again we may suppose $f\geq0$.  For $k=1,l=0$ (the other case is similar),  we can write 
\begin{align}
\label{lem3-1}
\nu\left\vert \nabla_{\Gamma_1}(\mathbb{E}_{\Gamma_0}  f^q)^{\frac{1}{q}} \right\vert
^q &\leq \nu\sum_{i\epsilon\Gamma_1}\left\vert \nabla_i (\mathbb{E}_{\{\sim i\}}f^q)^{\frac{1}{q}}\right\vert ^q \nonumber \\ 
&= \nu\sum_{i\epsilon\Gamma_1} \frac{1}{q^q}(\mathbb{E}_{\{\sim i\}}f^q)^{-\frac{q}{p}}  \left\vert\nabla_i( \mathbb{E}_{\{\sim i\}}  f^q)\right\vert ^q.
\end{align}

We will compute the terms in the sum on the right hand side of \eqref{lem3-1}.
For $i \in \Gamma_1$, we have
\begin{align*}
\nabla_i (\mathbb{E}_{\{\sim i\}} f^q) &=  q(\mathbb{E}_{\{\sim i\}} f^{q-1} \nabla_i f )-\sum_{ j\in\{\sim i\}}J_{i,j}\E_{\{\sim i\}}\left(f^q; \nabla_iV(x_i, x_j)\right) \\
\Rightarrow \left|\nabla_i (\mathbb{E}_{\{\sim i\}} f^q)\right| 
&\leq q\left( \E_{\{\sim i\}} f^q\right)^{1/p} \left( \E_{\{\sim i\}}|\nabla_if|^q\right)^{1/q} + J_0\left|\E_{\{\sim i\}}\left(f^q; W_i\right)\right|,
\end{align*}
so that
\begin{align*}
 \left|\nabla_i (\mathbb{E}_{\{\sim i\}} f^q)\right|^q &\leq 2^{q-1}q^q\left( \E_{\{\sim i\}} f^q\right)^{\frac{q}{p}} \left( \E_{\{\sim i\}}|\nabla_if|^q\right)\\
&\qquad + 2^{q-1}J^q_0\left|\E_{\{\sim i\}}\left(f^q; W_i\right)\right|^q,
\end{align*}
where $W_i = \sum_{j\in\{\sim i\}} \nabla_iV(x_i, x_j)$ as above.

We can use Lemma \ref{lem2} to bound the correlation in the second term.  Indeed, this gives
\begin{align*}
\left|\nabla_i (\mathbb{E}_{\{\sim i\}} f^q)\right|^q &\leq  \left(\E_{\{\sim i\}} f^q\right)^{\frac{q}{p}} \left(2^{q-1}q^q\E_{\{\sim i\}}|\nabla_if|^q + 2^{q-1}\kappa J^q_0\E_{\{\sim i\}}\left|\nabla_{\{\sim i\}} f\right|^q\right).
\end{align*}

Using this in \eqref{lem3-1} yields
\begin{align*}
\nu \left\vert \nabla_{\Gamma_1}(\mathbb{E}_{\Gamma_0}  f^q)^{\frac{1}{q}} \right\vert
^q &\leq \nu\sum_{i\in\Gamma_1}\left(2^{q-1}\E_{\{\sim i\}}|\nabla_if|^q + \frac{2^{q-1}}{q^q}\kappa J^q_0\E_{\{\sim i\}}\left|\nabla_{\{\sim i\}} f\right|^q\right) \\
& = 2^{q-1} \nu\left|\nabla_{\Gamma_1} f\right|^q + \frac{2^{q-1}}{q^q}\kappa J_0^q\nu\sum_{i\in\Gamma_1}\left|\nabla_{\{\sim i\}} f\right|^q \\
& = 2^{q-1} \nu\left|\nabla_{\Gamma_1} f\right|^q + \frac{2^{q+1}}{q^q}\kappa J_0^q\nu\left|\nabla_{\Gamma_0} f\right|^q.
\end{align*}
Finally, taking $J^q_0<\frac{q^q}{2^{q+1}\kappa}$ we see that
\[
\nu\left\vert \nabla_{\Gamma_1}(\mathbb{E}_{\Gamma_0}  f^q)^{\frac{1}{q}} \right\vert
^q \leq D\nu \left|\nabla_{\Gamma_1} f\right|^q + \eta\nu\left|\nabla_{\Gamma_0} f\right|^q,
\]
where $D=2^{q-1}$ and $\eta= \frac{2^{q+1}}{q^q}\kappa J_0^q <1$, as required.
\end{proof}

 \begin{lemma}
 \label{lem4}
 Suppose the local specification  $\{\mathbb{E}_{\Lambda}^{\omega}\}_{\Lambda\subset\subset \Z^N,\omega \in\Omega}$ satisfies {\rm \textbf{(H0)}} and {\rm \textbf{(H1)}}.  Then $\mathcal{P}^nf$ converges
 $\nu$-almost everywhere to $\nu f$,  where we recall that $\mathcal{P}=\mathbb{E}_{\Gamma_1}\mathbb{E}_{\Gamma_{0}}$.  In particular, $\nu$ is unique.
 \end{lemma}

 \begin{proof}
 
\noindent  
We will follow \cite{G-Z}. We have 
 \begin{align}\nonumber\nu\left|f- \mathbb{E}_{\Gamma_{1}}\mathbb{E}_{\Gamma_{0}} f\right|^q&\leq 2^{q-1}\nu\mathbb{E}_{\Gamma_{0}}\left|f- \mathbb{E}_{\Gamma_{0}} f\right|^q+2^{q-1}\nu\mathbb{E}_{\Gamma_{1}}\left|\mathbb{E}_{\Gamma_{0}}f- \mathbb{E}_{\Gamma_{1}}\mathbb{E}_{\Gamma_{0}} f\right|^q\\ &
 \leq2^{q-1}c_0\nu\left\vert \nabla_{\Gamma_0} f
\right\vert^q+2^{q-1}c_0\nu\left\vert \nabla_{\Gamma_1}(\mathbb{E}_{\Gamma_{0}}f
)\right\vert^q,
\end{align}           
since by \textbf{(H0)} and Remark \ref{tensorisation} both the measures $\mathbb{E}_{\Gamma_{0}}$ and $\mathbb{E}_{\Gamma_{1}}$ satisfy the $(SG_q)$ inequality with constant $c_0 = \frac{4c}{\log2}$ independant of the boundary conditions. If we use Lemma \ref{lem1} we get 
\[
\nu\left|f- \mathbb{E}_{\Gamma_{1}}\mathbb{E}_{\Gamma_{0}} f\right|^q \leq 2^{q-1}c_0\nu\left\vert \nabla_{\Gamma_0} f
\right\vert^q+2^{q-1}c_0(\tilde D\nu\left\vert \nabla_{\Gamma_1} f
\right\vert^q+\tilde\eta\nu\left\vert \nabla_{\Gamma_0} f
\right\vert^q)
\]
From the last inequality we obtain that for any   $n\in \mathbb{N}$,  
\begin{align*}
\nu\left|\mathcal{P}^nf- \mathcal{P}^{n+1} f\right|^q &\leq 2^{q-1}c_0\nu\left\vert \nabla_{\Gamma_0} \mathcal{P}^nf
\right\vert^q+2^{q-1}c_0\tilde\eta\nu\left\vert \nabla_{\Gamma_0} \mathcal{P}^nf
\right\vert^q \\
&= 2^{q-1}c_0(1+\tilde\eta)\nu\left\vert \nabla_{\Gamma_0} \mathcal{P}^nf\right\vert^2,
\end{align*}
using the fact that $\mathcal{P}^n$ does not depend on coordinates in $\Gamma_1$ by definition, so that $\nabla_{\Gamma_1}\mathcal{P}^n = 0$.  By repeated applications of Lemma \ref{lem1} we see that,

\begin{align*}
\nu\left|\mathcal{P}^nf- \mathcal{P}^{n+1} f\right|^q &\leq 2^{q-1}c_0(1+\tilde\eta)\tilde\eta^{2n-1}\nu\left|\nabla_{\Gamma_1}\E_{\Gamma_0}f\right|^q \\
&\leq 2^{q-1}c_0(1+\tilde\eta)\tilde\eta^{2n-1}\left(\tilde D\nu\left|\nabla_{\Gamma_1}f\right|^q + \tilde\eta\nu\left|\nabla_{\Gamma_0}f\right|^q\right).
\end{align*}

Since $\tilde\eta<1$, this clearly tends to zero as $n\to\infty$, so that the sequence $\{\mathcal{P}^n\}$ is Cauchy in $L^q(\nu)$.  Moreover, by the Borel-Cantelli lemma, the sequence
 $$\{ \mathcal{P}^n f-\nu \mathcal{P}^n f\}_{n \in \mathbb{N}}$$  
 converges $\nu-a.s$. The limit of $\mathcal{P}^n f-\nu \mathcal{P}^n f=\mathcal{P}^n f-\nu  f$    is therefore constant and hence identical to zero. 
 \end{proof}

 \subsection{Proof of Theorem \ref{thm1} }

\begin{proof}
Recall that we want to extend the $(LS_q)$ inequality from the single-site measures $\mathbb{E}_{i}^{\omega}$ to the Gibbs measure corresponding to the local specification $\{\mathbb{E}_{\Lambda}^{\omega}\}_{\Lambda\subset\subset \Z^2,\omega \in
\Omega}$ on the entire lattice (since we are taking $N=2$ for convenience).  As mentioned, to do so, we will follow the iterative method developed by B. Zegarlinski in  \cite{Ze2} and \cite{Ze}. 

Again without loss of generality, suppose $f\geq0$.  We can write 
 \begin{align}
 \label{thm1-1}
 \nonumber \nu \left(f^q \log\frac{f^q}{\nu f^q}\right)=&\nu\mathbb{E}_{\Gamma_0} \left(f^q \log\frac{f^q}{\mathbb{E}_{\Gamma_0} f^q}\right)+\nu\mathbb{E}_{\Gamma_{1}} \left(\mathbb{E}_{\Gamma_0}f^q \log\frac{\mathbb{E}_{\Gamma_0}f^q}{\mathbb{E}_{\Gamma_{1}}\mathbb{E}_{\Gamma_0} f^q}\right) \\ 
&\qquad + \nu \left(\mathbb{E}_{\Gamma_{1}}\mathbb{E}_{\Gamma_{0}}f^q \log\mathbb{E}_{\Gamma_{1}}\mathbb{E}_{\Gamma_{0}}f^q\right)- \nu\left(f^q \log \nu f^q\right).
\end{align}
As mentioned above, by \textbf{(H0)} and since the measures $\mathbb{E}_{\Gamma_0}$
and $\mathbb{E}_{\Gamma_1}$ are in fact product measures, we know that they both satisfy $(LS_q)$ with constant $c$ independent of the boundary conditions.  Using this fact in \eqref{thm1-1} yields

\begin{align}
\label{thm1-2}
\nonumber  \nu \left(f^q \log\frac{f^q}{\nu f^q}\right) \leq c\nu(\mathbb{E}_{\Gamma_0}&\left\vert \nabla_{\Gamma_0} f
\right\vert^q)+c\nu \mathbb{E}_{\Gamma_1}\left\vert \nabla_{\Gamma_1}(\mathbb{E}_{\Gamma_0} f^q)^{\frac{1}{q}}
\right\vert^q\\
& \qquad + \nu \left(\mathcal{P}f^q \log\mathcal{P}f^q\right)-\nu \left(f^q \log \nu f^q\right).
\end{align}
For the third term of \eqref{thm1-2} we can similarly write 
\begin{align}
\nonumber\nu \left(\mathcal{P}f^q \log \mathcal{P} f^q\right)&= \nu \mathbb{E}_{\Gamma_{0}}\left(\mathcal{P}f^q \log \frac{\mathcal{P} f^q}{\mathbb{E}_{\Gamma_{0}}\mathcal{P} f^q}\right)+\nu \mathbb{E}_{\Gamma_{1}}\left(\mathbb{E}_{\Gamma_{0}}\mathcal{P}f^q \log \frac{\mathbb{E}_{\Gamma_{0}}\mathcal{P} f^q}{\mathbb{E}_{\Gamma_{1}}\mathbb{E}_{\Gamma_{0}}\mathcal{P} f^q}\right)\\
& \qquad + \nonumber
\nu\left(\mathbb{E}_{\Gamma_{1}}\mathbb{E}_{\Gamma_{0}}\mathcal{P}f^q \log \mathbb{E}_{\Gamma_{1}}\mathbb{E}_{\Gamma_{0}}\mathcal{P}f^q \right).
\end{align}
If we use again the $(LS_q)$ inequality  for the measures $\mathbb{E}_{\Gamma_{k}},k=0,1$
we get
 \begin{equation}
 \label{thm1-3}
 \nu\left(\mathcal{P}f^q \log \mathcal{P} f^q\right)\leq c \nu\left\vert \nabla_{\Gamma_0}(\mathcal{P}f^q)^\frac{1}{q}
\right\vert^q+c \nu\left\vert \nabla_{\Gamma_1}(\mathbb{E}_{\Gamma_{0}} \mathcal{P}f^q)^\frac{1}{q}
\right\vert^q+\nu \left(\mathcal{P}^2f^q \log \mathcal{P}^2f^q\right).
\end{equation}
Working similarly for the last term $\nu \left(\mathcal{P}^2f^q \log \mathcal{P}^2f^q\right)$  of \eqref{thm1-3} and inductively for any term $\nu (\mathcal{P}^kf^q \log \mathcal{P}^kf^q)$, then after $n$ steps \eqref{thm1-2} and \eqref{thm1-3} will give
\begin{align}
\label{thm1-4}
\nonumber\nu \left(f^q \log\frac{f^q}{\nu f^q}\right) &\leq c \sum_{k=0}^{n-1} \nu \left\vert \nabla_{\Gamma_0}(\mathcal{P}^kf^q)^\frac{1}{q}\right\vert^q+c \sum_{k=0}^{n-1} \nu\left\vert \nabla_{\Gamma_1}(\mathbb{E}_{\Gamma_{0}} \mathcal{P}^kf^q)^\frac{1}{q}\right\vert^q \\
& \qquad + \nu \left(\mathcal{P}^n f^q \log\mathcal{P}^n f^q\right)-\nu \left(f^q \log \nu f^q\right).
 \end{align}
 
In order to deal with the first and second term on the right-hand side of \eqref{thm1-4} we will use Lemma \ref{lem3}.    If we apply inductively relationship \eqref{lem3ineq}, for any $k\in\mathbb{N}$ we obtain
\begin{equation}
\label{thm1-5}
\nu\left\vert \nabla_{\Gamma_0}(\mathcal{P}^kf^q)^\frac{1}{q}
\right\vert^q \leq \eta^{2k-1}D\nu\left\vert \nabla_{\Gamma_1} f
\right\vert^q+\eta^{2k}\nu\left\vert \nabla_{\Gamma_0} f
\right\vert^q
\end{equation}
and
\begin{equation}
\label{thm1-6}
\nu\left\vert \nabla_{\Gamma_1}(\mathbb{E}_{\Gamma_{0}}\mathcal{P}^kf^q)^\frac{1}{q}
\right\vert^q \leq \eta^{2k}D\nu\left\vert \nabla_{\Gamma_1} f
\right\vert^q+\eta^{2k+1}\nu\left\vert \nabla_{\Gamma_0} f
\right\vert^q.
\end{equation}

Using \eqref{thm1-5} and \eqref{thm1-6} in \eqref{thm1-4} we see that
\begin{align}
\label{thm1-7}
\nonumber\nu \left(f^q \log\frac{f^q}{\nu f^q}\right)& \leq cD\left(\eta^{-1}  + 1\right)\left(\sum_{k=0}^{n-1}\eta^{2k}\right)\nu\left\vert \nabla_{\Gamma_1} f\right\vert^q \\
& \qquad + c\left(1+\eta\right)\left(\sum_{k=0}^{n-1}\eta^{2k}\right)\nu\left\vert \nabla_{\Gamma_0} f\right\vert^q \nonumber \\ 
& \qquad + \nu \left(\mathcal{P}^n f^q \log \mathcal{P}^n f^q\right)-\nu (f^q \log \nu f^q).
\end{align}

By Lemma \ref{lem4} we have that $\lim_{n\to \infty}\mathcal{P}^nf^q=\nu f^q$,  $\nu-a.s$.   Therefore, taking the limit as $n\to\infty$ in \eqref{thm1-7} yields
\[
\nu \left(f^q \log\frac{f^q}{\nu f^q}\right) \leq cD\left(\frac{1}{\eta}+1\right)K\nu\left\vert \nabla_{\Gamma_1} f
\right\vert^q+c(1 + \eta)K\nu\left\vert \nabla_{\Gamma_0} f
\right\vert^q
\]
 where $K=\sum_{k=0}^{\infty}\eta^{2k} = \frac{1}{1-\eta^2}$ for  $\eta<1$.  Hence
 \[
\nu \left(f^q \log\frac{f^q}{\nu f^q}\right) \leq C\nu|\nabla f|^q
\]
for $C = \max\left\{ cD\left(\frac{1}{\eta}+1\right)K, c(1 + \eta)K\right\}$, as required.
\end{proof}


\begin{thebibliography}{0}


\bibitem{A-B-C-F-G-M-R-S}
{\sc C.~An{\'e}, S.~Blach{\`e}re, D.~Chafa{\"\i}, P.~Foug{\`e}res, I.~Gentil,
  F.~Malrieu, C.~Roberto, and G.~Scheffer}, {\em Sur les in{\'e}galit{\'e}s de
  Sobolev logarithmiques}, no.~10 in Panoramas et Synth{\`e}ses, Soc. Math.
  France, Paris, 2000.

\bibitem{Bakry}
{\sc D.~Bakry}, {\em L'hypercontractivit\'e et son utilisation en th\'eorie des
  semigroups}, no.~1581 in Lecture Notes in Math., Springer, 1994, pp.~1--114.

\bibitem{B-B-B-C}
{\sc D.~Bakry, F.~Baudoin, M.~Bonnefont, and D.~Chafa{\"\i}}, {\em On gradient
  bounds for the heat kernel on the {H}eisenberg group}, J. Funct. Anal., 255
  (2008), pp.~1905--1938.

\bibitem{Beals}
{\sc R.~Beals, B.~Gaveau, and P.~C. Greiner}, {\em {H}amilton-{J}acobi theory
  and the heat kernel on {H}eisenberg groups}, J. Math. Pures. Appl., 79
  (2000), pp.~633--689.

\bibitem{Bel}
{\sc A.~Bella{\"\i}che}, {\em The tangent space in sub-{R}iemannian geometry},
  in Sub-Riemannian Geometry, A.~Bella{\"\i}che and J.~J. Risler, eds.,
  Progress In Mathematics, Birkh{\"a}user, 1996, pp.~4--84.

\bibitem{B-L}
{\sc S.~Bobkov and M.~Ledoux}, {\em From {B}runn-{M}inkowski to
  {B}rascamp-{L}ieb and to logarithmic {S}obolev inequalities}, Geom. Funct.
  Anal., 10 (2000), pp.~1028--1052.

\bibitem{B-Z}
{\sc S.~Bobkov and B.~Zegarlinski}, {\em Entropy Bounds and Isoperimetry},
  no.~829 in Mem. Amer. Math. Soc., Amer. Math. Soc., 2005.

\bibitem{B-H1}
{\sc T.~Bodineau and B.~Helffer}, {\em The log-{S}obolev inequalities for
  unbounded spin systems}, J. Funct. Anal., 166 (1999), pp.~168--178.

\bibitem{B-L-U}
{\sc A.~Bonfiglioli, E.~Lanconelli, and F.~Uguzzoni}, {\em Stratified Lie
  Groups and Potential Theory for their Sub-Laplacians}, Springer Monographs in
  Mathematics, Springer, 2007.

\bibitem{D-S}
{\sc J.~Deuschel and D.~Stroock}, {\em Large Deviations}, vol.~137 of Pure and
  Applied Mathematics, Academic Press, 1989.

\bibitem{Fed}
{\sc P.~Federbush}, {\em Partially Alternate Derivation of a Result of Nelson}, J. of Math Physics 10 (1969), pp. 50-52.

\bibitem{Grom}
{\sc M.~Gromov}, {\em {C}arnot-{C}arath{\'e}odory spaces seen from within}, in
  Sub-Riemannian Geometry, A.~Bella{\"\i}che and J.~J. Risler, eds., Progress
  In Mathematics, Birkh{\"a}user, 1996, pp.~85--324.

\bibitem{G}
{\sc L.~Gross}, {\em Logarithmic {S}obolev inequalities}, Amer. J. Math., 97
  (1975), pp.~1061--1083.

\bibitem{G-Z}
{\sc A.~Guionnet and B.~Zegarlinski}, {\em Lectures on logarithmic {S}obolev
  inequalites}, in S{\'e}minaire de Probabilit{\'e}s, XXXVI, no.~1801 in
  Lecture Notes in Math., Springer-Verlag, 2003, pp.~1--134.

\bibitem{H-Z}
{\sc W.~Hebisch and B.~Zegarlinski}, {\em Coercive inequalities on metric
  measure spaces}, J. Funct. Anal., 258 (2010), pp.~814--851.

\bibitem{Helffer}
{\sc B.~Helffer}, {\em Semiclassical Analysis, Witten Laplacians, and
  Statistical Mechanics}, Partial Differential Equations and Applications,
  World Scientific, 2002.

\bibitem{Ledconc}
{\sc M.~Ledoux}, {\em Concentration of measure and logarithmic {S}obolev
  inequalities}, in S{\'e}minaire de Probabilit{\'e}s, XXXIII, no.~1709 in
  Lecture Notes in Math., Springer-Verlag, 1999, pp.~120--216.

\bibitem{Led}
\leavevmode\vrule height 2pt depth -1.6pt width 23pt, {\em Logarithmic
  {S}obolev inequalities for unbounded spin spin systems revisited}, in
  S{\'e}minaire de Probabilit{\'e}s, XXXV, no.~1755 in Lecture Notes in Math.,
  Springer-Verlag, 2001, pp.~167--194.

\bibitem{Li}
{\sc H.~Q. Li}, {\em Estimation optimale du gradient du semi-groupe de la
  chaleur sur le groupe de {H}eisenberg}, J. Funct. Anal., 236 (2006),
  pp.~369--394.

\bibitem{L-Z}
{\sc P.~Lugiewicz and B.~Zegarlinski}, {\em Coercive inequalities for
  {H}\"ormander type generators in infinite dimensions}, J. Funct. Anal., 247
  (2007), pp.~438--476.

\bibitem{Marton}
{\sc K.~Marton}, {\em An explicit bound on the Logarithmic Sobolev constant for weakly dependent random variables}, arXiv:math/0605397  (2007).

\bibitem{Monti}
{\sc R.~Monti}, {\em {S}ome properties of {C}arnot-{C}arath{\'e}odory balls in
  the {H}eisenberg group}, Rend. Mat. Acc. Lincei, 11 (2000), pp.~155--167.

\bibitem{O-R}
{\sc F.~Otto and M.~G. Reznikoff}, {\em A new criterion for the logarithmic
  {S}obolev inequality and two applications}, J. Funct. Anal., 243 (2007),
  pp.~121--157.

\bibitem{Ioannis1}
{\sc I.~Papageorgiou}, {\em The logarithmic {S}obolev inequality in infinite
  dimensions for unbounded spin systems on the lattice with non-quadratic
  interactions}, Markov Processes Relat. Fields 16 (2010), 447-484.

\bibitem{Ioannis2}
\leavevmode\vrule height 2pt depth -1.6pt width 23pt, {\em Pertubing the
  logarithmic {S}obolev inequality for unbounded spin systems on the lattice
  with non-quadratic interactions}, arXiv:0901.1482v1,  (2009).

\bibitem{Ro}
{\sc O.~Rothaus}, {\em Analytic inequalities, isoperimetric inequalities and
  logarithmic {S}obolev inequalities}, J. Funct. Anal., 64 (1985),
  pp.~296--313.

\bibitem{S-Z1}
{\sc D.~Stroock and B.~Zegarlinski}, {\em The logarithmic {S}obolev inequality
  for continuous spin systems on a lattice}, J. Funct. Anal., 104 (1992),
  pp.~299--326.

\bibitem{Var}
{\sc N.~T. Varopoulos, L.~Saloff-Coste, and T.~Coulhon}, {\em Analysis and
  Geometry on Groups}, no.~100 in Cambridge Tracts in Mathematics, CUP, 1992.

\bibitem{Yo2}
{\sc N.~Yosida}, {\em Application of log-{S}obolev inequality to the stochastic
  dynamics of unbounded spin systems on the lattice}, J. Funct. Anal., 173
  (2000), pp.~74--102.

\bibitem{Ze2}
{\sc B.~Zegarlinski}, {\em On log-{S}obolev inequalities for infinite lattice
  systems}, Lett. Math. Phys., 20 (1990), pp.~173--182.

\bibitem{Ze}
\leavevmode\vrule height 2pt depth -1.6pt width 23pt, {\em The strong decay to
  equilibrium for the stochastic dynamics of unbounded spin systems on a
  lattice}, Comm. Math. Phys., 175 (1996), pp.~401--432.



\end{thebibliography}
\end{document}